\documentclass[11pt,reqno]{amsart}
\usepackage{amscd,amsmath,amsopn,amssymb,amsthm,multicol}
\usepackage[backref=page,breaklinks=true,colorlinks=true,linkcolor=blue,citecolor=blue,urlcolor=blue]{hyperref}

\DeclareMathOperator{\orth}{orth}
\DeclareMathOperator{\Inn}{Inn}

\DeclareMathOperator{\ad}{ad}

\DeclareMathOperator{\Rad}{Rad}
\DeclareMathOperator{\Ad}{Ad}

\DeclareMathOperator{\Aut}{Aut}

\DeclareMathOperator{\Isom}{Isom}

\DeclareMathOperator{\Lev}{Lev}

\DeclareMathOperator{\SL}{SL}
\DeclareMathOperator{\SO}{SO}
\DeclareMathOperator{\Nil}{Nil}
\DeclareMathOperator{\nc}{nc}
\DeclareMathOperator{\cop}{c}
\def\mfg{\mathfrak{g}}
\def\mfm{\mathfrak{m}}
\def\mfh{\mathfrak{h}}
\def\mfq{\mathfrak{q}}
\def\mfl{\mathfrak{l}}
\def\mff{\mathfrak{f}}
\def\mfw{\mathfrak{w}}
\def\mfu{\mathfrak{u}}
\def\mfv{\mathfrak{v}}
\def\mfk{\mathfrak{k}}

\def\mfp{\mathfrak{p}}
\def\mfz{\mathfrak{z}}

\def\mfs{\mathfrak{s}}

\def\mfn{\mathfrak{n}}

\def\mfo{\mathfrak{o}}

\def\ov{\overline}
\def\ovphi{\overline{\Phi}}
\def\ovla{\overline{L}_a}
\def\ovphi{\overline{\Phi}}
\def\ovra{\overline{R}_a}
\def\ovia{\overline{I}_a}
\def\hg{\widehat{G}}
\def\hmfg{\widehat{\mfg}}

\def\la{\langle}
\def\ra{\rangle}
\def\ip{\la \cdot,\,\cdot\ra}
\def\R{\mathbf{R}}
\def\t{\widetilde}
\def\lg{\Lev(\mfg)}

\def\lgnc{\Lev(\mfg)_{\nc}}
\def\lGnc{\Lev(G)_{\nc}}
\def\lgc{\Lev(\mfg)_{\cop}}

\def\wt{\widetilde}

\renewenvironment{proof}[1][Proof]{\textbf{#1.} }
{\ \rule{0.5em}{0.5em}}

\numberwithin{equation}{section}
\newtheorem{theorem}{Theorem}[section]
\newtheorem{prop}{Proposition}[section]
\newtheorem{lemma}{Lemma}[section]
\newtheorem{corollary}{Corollary}[section]

\theoremstyle{definition}
\newtheorem{definition}{Definition}[section]
\newtheorem{remark}{Remark}[section]

\newtheorem{example}{Example}[section]
\newtheorem{notation}{Notation}[section]
\newtheorem{notarem}{Notation and Remarks}[section]

\begin{document}

\title
[G.O. Structures on $\R^n$] {Geodesic Orbit Riemannian Structures on $\R^n$}

\author{Carolyn~S.~Gordon}
\address{C.\,S. Gordon\newline
Dartmouth College, Hanover, NH 03755-1808, USA}
\email{carolyn.s.gordon@dartmouth.edu}
\author{Yuri\u{i}~G.~Nikonorov}
\address{Yu.\,G. Nikonorov\newline
Southern Mathematical Institute of the Vladikavkaz \newline
Scientific Centre of the Russian Academy of Sciences,\newline
Vladikavkaz, Markus str. 22, 362027, Russia}
\email{nikonorov2006@mail.ru}

\begin{abstract}

A geodesic orbit manifold is a complete Riemannian manifold all of whose geodesics are orbits of one-parameter groups of isometries.
We give both a geometric and an algebraic characterization of geodesic orbit manifolds that are diffeomorphic to~$\R^n$.
Along the way, we establish various structural properties of more general geodesic orbit manifolds.

\vspace{2mm} \noindent 2010 Mathematical Subject Classification:
53C20, 53C25, 53C35.

\vspace{2mm} \noindent Key words and phrases: homogeneous Riemannian manifolds,  homogeneous spaces,
geodesic orbit spaces, geodesic orbit manifolds.
\end{abstract}

\maketitle

\thispagestyle{empty}

\section{Introduction}

A complete Riemannian manifold $(M,g)$ is said to be a \emph{geodesic orbit} manifold, abbreviated {\it G.O. manifold},    if every
geodesic is an orbit of a one-parameter group of isometries.  G.O. manifolds are necessarily homogeneous.  Among their nice geometric properties,
every G.O. manifold is a ``D'Atri space'', i.e., the local geodesic symmetries are volume preserving up to sign (see \cite{KPV}).   While the G.O.
property results in strong restrictions both on the structure of the isometry group and on the geometry, the class of G.O. manifolds is nonetheless
large enough to admit many interesting classes of homogeneous Riemannian metrics.
We refer to \cite{KV}, \cite{Arv}, and \cite{Nik2016} for expositions on general properties of geodesic
orbit Riemannian manifolds and historical surveys.

The primary goal of this article is to study both the geometry and the symmetry properties of G.O. manifolds that are diffeomorphic to $\R^n$.
Our main results, included in Theorems ~\ref{thm.solv} and \ref{thm.submers}, are summarized in Theorem~\ref{thm.main}.
First recall that a simply-connected nilpotent Lie group with a left-invariant metric is called a \emph{Riemannian nilmanifold}.
By a result of \cite{Gor96}, every G.O. nilmanifold has step size at most two.

\begin{theorem}\label{thm.main} Let $(M,g)$ be a G.O. manifold diffeomorphic to $\R^n$.  Then:
\begin{enumerate}
\item $(M,g)$ is the total space of a Riemannian submersion $\pi:M\to P$, where the base space $P$ is a Riemannian symmetric space of noncompact type.  The fibers are totally geodesic and are isometric to a G.O. nilmanifold $(N,g)$ of step size at most two.

\item $M$ admits a simply-transitive solvable group of isometries of the form $S\times N$ where $S$ is an Iwasawa subgroup of a semisimple Lie group and
$N$ is the group in {\rm(1)}.

\end{enumerate}

\end{theorem}

Some of our additional results include:
\begin{itemize}
\item We correct an error in \cite{Gor96}, Theorem 1.15.   Let $(M,g)$ be a G.O. manifold and let $\Isom(M,g)$ be its full isometry group.  In \cite{Gor96},
the first author incorrectly asserted that if $G$ is any transitive subgroup of $\Isom(M,g)$ with the property that every geodesic in $M$ is an orbit of
a one-parameter subgroup of $G$, then $\Lev(G)\Nil(G)$ acts transitively on $M$, where $\Lev(G)$ and $\Nil(G)$ are a semisimple Levi factor and
the nilradical of $G$, respectively.   The second author~\cite{Nik2016} recently gave a~counterexample.
Theorem~\ref{thm.nil} proves a slightly weakened version of the statement in \cite{Gor96}.

\item Proposition~\ref{proalgonilp1} gives a new result concerning G.O. nilmanifolds.
\end{itemize}

The remainder of these introductory remarks partially reviews the history of G.O. manifolds and provides some context for our results.

The naturally reductive homogeneous Riemannian manifolds (among which are all the Riemannian symmetric spaces) were the first large class of Riemannian manifolds known to have the geodesic orbit property.  Recall that a Riemannian manifold $(M,g)$ is
{\it naturally reductive} if it admits a transitive Lie group $G$ of isometries with a bi-invariant
pseudo-Riemannian metric $g_0$ that induces the metric $g$ on $M = G/H$; see  \cite{Bes}.  In 1984, A.~Kaplan~\cite{Kap} gave the first example of a non-naturally reductive Riemannian manifold that exhibits the geodesic orbit property.
The terminology ``G.O. manifold'' was first introduced in 1991 by O.~Kowalski and L.~Vanhecke~\cite{KV}.   Since that time, a number of additional interesting
classes of metrics have been shown to have the G.O. property.  J.~Berndt, O.~Kowalski, and L.~Vanhecke~\cite{BKV} showed that all \emph{weakly symmetric spaces}
are G.O. manifolds.  The weakly symmetric spaces, introduced by A.~Selberg~\cite{S}, have the defining property that
any two points can be interchanged by
an isometry; these manifolds are closely related to spherical spaces, commutative
spaces, and Gel'fand pairs. (See \cite{AV, Yakimova, W1}.)
Another subclass of G.O. manifolds are the {\it generalized normal homogeneous Riemannian manifolds}, also called
{\it $\delta$-homogeneous manifolds}.
All metrics from this subclass are of non-negative sectional curvature
(see \cite{BerNik, BerNik3, BerNik2012}).
The {\it Clifford--Wolf homogeneous
Riemannian manifolds} are among the generalized normal homogeneous manifolds~\cite{BerNikClif}.

On the other hand, the G.O. condition is quite demanding.  For example, the only G.O. manifolds of negative Ricci curvature are the
noncompact Riemannian symmetric spaces~\cite{Gor96}.   G.O. metrics have been classified among all metrics in various settings,
e.g. on spheres~\cite{Nik2013}, on flag manifolds~\cite{AA} and, more generally, on manifolds with positive Euler characteristic \cite{AN},
and among metrics that fiber over irreducible symmetric spaces~\cite{Ta}.  The article~\cite{Gor96} shows that the G.O. condition imposes
a number of structural conditions on the full isometry group of the metric.

Various constructions of geodesics that are orbits of one-parameter isometry groups, and interesting properties of geodesic orbit metrics
can be found in \cite{CHNT,DuKoNi,Nik2013n,Tam,Zil96} and in the references therein.
For generalizations to pseudo-Riemannian manifolds, see
\cite{Barco,ChenWolf2012,Du2010} and the references therein; for generalizations to Finsler manifolds, see
 \cite{YanDeng}.

\bigskip

The paper is organized as follows:  In Section~\ref{krss}, we provide the needed background on Riemannian homogeneous spaces.  We turn to G.O. manifolds and prove our main results in Section~\ref{mrss}.

\medskip

\section{Background on Riemannian homogeneous spaces}\label{krss}

Throughout both this section and the next, all manifolds will be assumed to be connected.  In this section, we review some elementary properties of isometry groups of homogeneous Riemannian manifolds.

Let $(M, g)$ be a connected homogeneous Riemannian manifold and let $\Isom(M,g)$ be its full isometry group.
Let $G$ be a closed connected transitive subgroup of $\Isom(M,g)$. Since $\Isom(M,g)$ and thus $G$ act effectively on $M$,
the isotropy subgroup $H$ of $G$ at a chosen base point $o$ of $M$
has trivial  intersection with the center $C(G)$.  The mapping $G\to M$ given by $a\mapsto a(o)$ induces a diffeomorphism $G/H\to M$.  Via this identification, we may view $g$ as a Riemannian metric on $G/H$.  The metric is left-invariant; i.e., for $a\in G$, the left-translation $L_a: G\to G$ induces an isometry $(G/H,g)\to (G/H,g)$.  We denote this isometry by $\ov{L_a}:G/H\to G/H$.

\begin{notarem}\label{autorth}
(1) We will always denote the Lie algebra of a Lie group by the corresponding fraktur letter.

(2) For $G/H$ as above, the Killing form $B_{\mfg}$ is negative-definite on $\mfh$ and thus we can write
\begin{equation}\label{reductivedecomposition}
\mathfrak{g}=\mathfrak{h}\oplus \mathfrak{m},
\end{equation}
vector space direct sum,
where $\mathfrak{m}$ is the orthogonal complement of $\mfh$ relative to the Killing form $B_{\mfg}$.   Observe that $\mathfrak{m}$ is an $\Ad_G(H)$-invariant subspace of  $\mfg$ and can be identified with the tangent
space $T_oM$ at the base point $o = eH$.
The Riemannian metric $g$ corresponds to an $\Ad(H)$-invariant inner product $g_{eH}=\ip$ on $\mfm$.   Let $O(\mfm, \ip)$ denote the group of orthogonal transformations of $\mathfrak{m}$.

(3) If $\Phi$ is an automorphism of $G$ that normalizes $H$, then $\Phi$ induces a well-defined diffeomorphism $\ov{\Phi}$ of $G/H$ (and hence of $M$) by $\ov{\Phi}(aH)=\Phi(a)H$.   We write
$$\Aut_{\orth}(G/H,g)=\{\ovphi: \Phi\in\Aut(G), \,\Phi(H)=H, \,\mbox{and}\,\ovphi_{*eH}\in O(\mfm, \ip )\}$$
where $\ovphi_*$ is the differential of $\ovphi$.

\end{notarem}

For completeness, we include the proof of the following elementary and well-known result.

\begin{lemma}\label{lem.autorth}
$\Aut_{\orth}(G/H,g)$ consists of isometries of $(G/H,g)$ and forms the isotropy subgroup at $o$ of the normalizer of $G$ in $\Isom(G/H,g)$.
 \end{lemma}

 \begin{proof}
The automorphism property of $\Phi$ implies that
$$\ovphi\circ \ovla=\ov{L}_{\Phi(a)}\circ\ovphi$$
for all $a\in G$.    Thus
$$\ovphi_{*aH}=(\ov{L}_{\Phi(a)})_{*eH}\circ\ovphi_{*eH}\circ (\ovla)^{-1}_{*aH}.$$
Since the left translations are isometries, it follows that $\ovphi$ preserves the  metric $g$, i.e., $\ovphi\in\Isom(G/H,g)$.   Trivially, $\ovphi$ lies in the isotropy group at $o=eH$.

If $\tau\in \Isom(G/H,g)$ normalizes $G$ and fixes $o$, then conjugation by
$\tau$ in $\Isom(G/H,g)$ restricts to an automorphism $\Phi$ of $G$ normalizing $H$, and one easily checks that $\tau=\ovphi$.
\end{proof}
\smallskip

We will be primarily interested in those $\ovphi\in \Aut_{\orth}(G/H,g)$ such that $\Phi$ is an inner automorphism.   For $a\in G$, let $I_a=L_a\circ R_a\in\Inn(G)$ be the associated inner automorphism, where $R_a(x)=xa^{-1}$.  Suppose that $a$ normalizes $H$ so that $\ovia$ is well-defined.  Since $g$ is left-invariant, we have that $\ovia\in \Isom(G/H,g)$ if and only if $\ovra\in \Isom(G/H,g)$.  If $h\in H$, then $\ov{R}_h$ is trivial and thus
$\ov{I}_h=\ov{L}_h$.  In particular, $\Aut_{\orth}(G/H,g)\cap G =H$, under the identification of elements $a$ of $G$ with isometries $\ovla$ of $G/H$.   We now examine those inner automorphisms that give rise to elements of $\Aut_{\orth}(G/H,g)$ that are not already in $G$.

 Let $N_G(H)$ denote the normalizer of $H$ and let
\begin{equation}\label{W}W=\{a\in N_G(H): \ovia\in \Aut_{\orth}(G/H,g)\}.\end{equation}

\begin{lemma}\label{lem.norm}\hfill\\
We use Notation~\ref{autorth}.
\begin{enumerate}
\item The Lie algebra $N_{\mfg}(\mfh)$ of $N_G(H)$ is the normalizer of $\mfh$ in $\mfg$ and satisfies
 $$N_{\mfg}(\mfh)=\mfh+ (N_{\mfg}(\mfh)\cap \mfm)=\mfh+C_{\mfm}(\mfh)$$
where $C_{\mfm}(\mfh)$ is the centralizer of $\mfh$ in $\mfm$.

\item $[C_{\mfm}(\mfh), \mfm] \subset \mfm$ and $C_{\mfm}(\mfh)$ is a subalgebra of $\mfg$.

\item The group $W$ in Equation~(\ref{W}) has Lie algebra\\
$\mfw=\mfh +(\mfw\cap \mfm)=\mfh+ \{Y\in C_{\mfm}(\mfh): \ad(Y)|_{\mfm}\mbox{\,is skew-symmetric}\}.$  In particular, $\mfw$ is compactly embedded in $\mfg$.

\item $\mfw\cap \mfm$ is a subalgebra of $\mfg$ and can be decomposed into a direct sum of subalgebras $\mfw\cap \mfm= C(\mfg)\oplus\mff$ for some $\mff$, where $C(\mfg)$ is the center of $\mfg$.
\end{enumerate}
\end{lemma}

\begin{proof}
{\rm 1)} Since $\mfm$ is $\Ad_G(H)$-invariant, we have $[\mfh,\mfm]<\mfm$. Hence $N_{\mfm}(\mfh)=C_{\mfm}(\mfh)$ and (1) follows.

{\rm 2)} Let $X\in C_{\mfm}(\mfh)$ and $Y\in \mfm$.  For all $A\in\mfh$, we have
$$0=B_\mfg([X,A], Y)=-B_\mfg(A, [X,Y]).$$
Recalling Notation~\ref{autorth}, it follows that $[C_{\mfm}(\mfh), \mfm] \subset \mfm$.   This proves the first statement and the second statement is then immediate.

{\rm 3)} By Lemma~\ref{lem.autorth}, we have $$W=\{a\in N_G(H): \ovia\in O(\mfm,\ip)\}$$ and thus
$$\mfw=\{X\in N_{\mfg}(\mfh): \ad(X)|_{\mfm}\mbox{\,is skew-symmetric}\}.$$
(3) now follows from (1).

 {\rm 4)} The fact that $\mfw\cap\mfm$ is a subalgebra of $\mfg$ is immediate from (2).  By the definition of $\mfm$ in Notation~\ref{autorth}(ii) and the fact that the Killing form is negative-definite on $\mfh$, we see that $C(\mfg)<\mfm$.    Since $\mfw\cap\mfm$ is compactly embedded in $\mfg$, it is completely reducible.   Thus there exists a $(\mfw\cap\mfm)$-ideal complementary to the ideal $C(\mfg)$.
 \end{proof}

\begin{prop}\label{prop.gf}  In the notation of Lemma~\ref{lem.norm}, let $F$ be the connected subgroup of $W$ with Lie algebra $\mff$, and let $D:=F\cap C(G)$.   Let
$$
\ov{\Inn}_F=\{\ovia: a\in F\}< \Aut_{\orth}(G/H,g).
$$
Then:
\begin{enumerate}
\item The group $G\times F$ acts isometrically on $(G/H,g)$ via the action
$$
\rho(a,b)=\ovla(a)\circ\ovra(b).
$$
Letting $\Delta:F\to G\times F$ be the diagonal embedding $\Delta(x)=(x,x)$,
the effective kernel of the action $\rho$ is discrete and is given by $\Delta(D)=\{(d,d):d\in D\}$.
Thus $\rho$ gives rise to an effective action, again denoted by $\rho$ of $(G\times F)/\Delta(D)$ on $(G/H,g)$.

\item $(G\times F)/\Delta(D)=G\rtimes \ov{\Inn}_F <\Isom(G/H,g)$.

\item  Let $i:G\to G\times F$ be the inclusion $x\mapsto
(x,e)$.  The isotropy subgroup of $(G\times F)$ at $o=eH$ is given by $i(H)\Delta(F)$.  {\rm(}Note that for $a\in F$, we have $\rho(a,a)=\ovia$.{\rm)}

\item Let $\mfu$ be a subalgebra of $\mfg$ and $\mfv$ a subalgebra of $\mff$ such that $\mfg=\mfh +\mfu +\mfv$
{\rm(}with possibly non-trivial intersections{\rm)} and let $U$ and $V$ be  the corresponding connected subgroups of $G$ and $F$, respectively.
Then the subgroup $U\times V$ of $G\times F$ acts transitively on $G/H$ via the action $\rho$.
\end{enumerate}
\end{prop}

\begin{proof} Statements (1)--(3) are standard and follow easily from Lemma~\ref{lem.norm}.   For (4), the subgroup $\hg:=G\times V$ of $G\times F$ acts transitively on $G/H$ with isotropy subgroup $\widehat{H}:=i(H)\Delta(V)$, where $\Delta(V)=\{(a,a): a\in V\}$.  The corresponding Lie algebras satisfy $\hmfg=\mfg\times \mfv$ and  $\widehat{\mfh}=(\mfh\times \{0\}) + \Delta_*(\mfv)$.   Since  $(\mfv\times \{0\}) + \Delta_*(\mfv) = (\{0\}\times \mfv) +\Delta_*(\mfv)$, we have
$$
(\mfu\times \mfv )+\widehat{\mfh} =((\mfh + \mfu +\mfv)\times \{0\}) +\Delta_*(\mfv) =\hmfg.
$$
It follows that $\hg =(U\times V)\widehat{H}$ and thus $U\times V$ acts transitively on $\hg/\widehat{H}=G/H$.
  \end{proof}

We now consider the case in which the Lie algebra $\mfw$ coincides with the full normalizer of $\mfh$ in $\mfg$.  Equivalently, $C_{\mfm}(\mfh)=C(\mfg)+\mff$.

\begin{corollary}\label{cor.nilrad} We use the notation of Lemma~\ref{lem.norm} and Propostion~\ref{prop.gf}.
Suppose that $\ad(Y)|_{\mfm}$ is skew-symmetric for all $Y\in C_{\mfm}(\mfh)$.  Let $\Lev(G)$ be a semisimple Levi factor of $G$ and
$\Nil(G)$ the nilradical of $G$.   Set $U=\Lev(G)\Nil(G)$.  {\rm(}Since all semisimple Levi factors are conjugate via elements of the nilradical,
$U$ is independent of the choice of $\Lev(G)$.{\rm)}  Letting $C(F)$ denote the center of $F$,
then $U\times C(F)$ acts transitively and isometrically via $\rho$ on $(G/H,g)$.
The solvable radical of $U\times C(F)$ is the nilpotent Lie group $\Nil(G) \times C(F)$.
\end{corollary}

  \begin{proof}
Since $\ad_{\mfg}(\mfh)$ acts fully reducibly on $\mfm$, we have $\mfm =C_\mfm(\mfh)+[\mfh,\mfm] $.   Thus the hypothesis and Lemma~\ref{lem.norm} imply that
\begin{equation}\label{mfg}\mfg= \mfh +C(\mfg) +\mff + [\mfh, \mfm].\end{equation}
For any Lie algebra $\mfk$, we have $[\mfk,\mfk]< \Lev(\mfk) +\Nil(\mfk)$ where $\Lev(\mfk)$ and $\Nil(\mfk)$ are a semisimple Levi
factor and the  nilradical of $\mfu$, respectively.   Since $\mff$ is a compact Lie algebra, we have $\mff =[\mff,\mff] +C(\mff)< [\mfg,\mfg] +C(\mff)$.
The center $C(\mfg)$ lies in $\Nil(\mfg)$.  Thus Equation~(\ref{mfg}) implies that
   $$\mfg=\mfh +\Lev(\mfg) + \Nil(\mfg) + C(\mff).$$
The corollary now follows from Proposition~\ref{prop.gf}.
  \end{proof}

  We conclude this section by reviewing some structural results.

\begin{lemma}\label{lem.compact} Let $\mfg$ be a Lie algebra and let $\mfl$ be a maximal compactly embedded subalgebra of $\mfg$.
There exists a Levi decomposition $\mfg =\Lev(\mfg) +\Rad(\mfg)$ such that:
\begin{enumerate}

\item $\mathfrak{l} = (\mathfrak{l}\cap \Rad(\mfg)) \oplus (\mathfrak{l}\cap \Lev(\mfg))$;

\item $\mathfrak{l}\cap \Rad(\mfg)$ is abelian and commutes with $\Lev(\mfg)$.   Thus \\
$\Lev(\mfg)+\mfl=\Lev(\mfg)+(\mathfrak{l}\cap \Rad(\mfg))$ is a reductive subalgebra of $\mfg$.

\item $[\mathfrak{l},\mathfrak{l}] \subset \Lev(\mfg)$ and $\mathfrak{l} \cap \Lev(\mfg)$ is a maximal compact subalgebra in $\Lev(\mfg)$.
\end{enumerate}
\end{lemma}

See e.~g. Lemma 14.3.3 in~\cite{HilNeeb}.

\begin{definition}\label{nota.compat} Let $G/H$ be a connected Riemannian homogeneous manifold.
We will say that a semisimple Levi factor $\Lev(\mfg)$ is \emph{compatible} with $\mfh$ if $\lg$ satisfies the conclusions of Lemma~\ref{lem.compact}
relative to some maximal compactly embedded subalgebra $\mfl$ of $\mfg$ containing~$\mfh$.
We will also say that the corresponding Levi factor $\Lev(G)$ of $G$ is \emph{compatible} with $H$ and that $(\Lev(\mfg), \mfh, \mfl)$ is a \emph{compatible triple}.    \end{definition}

\section{The main results}\label{mrss}

As discussed in the Introduction, a Riemannian manifold is said to be a G.O. manifold if every geodesic is the orbit of a one-parameter group of isometries.
In the literature, there is often a distinction made between the language G.O. \emph{manifold} and
G.O. \emph{space}.

\begin{notation}\label{nota.go} A Riemannian homogeneous space expressed in the form $(G/H,g)$, where~$g$ is a left-invariant Riemannian metric on $G/H$,
is said to be a G.O. \emph{space} if every geodesic is the orbit of a one-parameter subgroup of $G$.
\end{notation}

Every G.O. manifold $(M,g)$ has a (not necessarily unique) realization $G/H$ as a G.O. space.   One can choose $G=\Isom(M,g)$ for example.

 We use the notation introduced in the previous section.  In particular, given a homogeneous space expressed in the form $(G/H,g)$ where $g$ is a left-invariant Riemannian metric on $G/H$, we let $\mfm$ denote the orthogonal complement of $\mfh$ in $\mfg$ relative to the Killing form of $\mfg$.
By $[\cdot, \cdot]$ we denote the Lie bracket in $\mathfrak{g}$, and by
$[\cdot, \cdot]_{\mathfrak{m}}$ its $\mathfrak{m}$-component with respect to the decomposition $\mfg=\mfh+\mfm$.  We recall a well-known criterion for $(G/H,g)$
to be a~geodesic orbit space.

\begin{lemma}[\cite{KV}]\label{G.O.-criterion}
$(G/H,g)$ is a geodesic orbit space if and
only if for each $X \in \mathfrak{m}$ there exists $Z \in \mathfrak{h}$ such that
$\la [X+Z,Y]_{\mathfrak{m}},X\ra =0$ for all $Y\in \mathfrak{m}$ where $\ip$ is the inner product on $\mfm$ defined by $g$.
\end{lemma}

 \begin{lemma}\label{lem.skew}  Suppose  that $(G/H,g)$ is a geodesic orbit space.  Then for all $Y\in C_\mfm(\mfh)$, the endomorphism $\ad_\mfg(Y)$ leaves $\mfm$ invariant and $\ad_\mfg(Y)|_{\mfm}$ is skew-symmetric .
 \end{lemma}

  Lemma~\ref{lem.skew} is immediate from Lemma~\ref{G.O.-criterion} since $[Z,Y]=0$ in the notation of the lemmas.

  The following well-known lemma is also an elementary consequence of Lemma~\ref{G.O.-criterion}.

\begin{lemma}\label{lem.subman} Let $(G/H,g)$ be a Riemannian G.O. space, and let $Q$ be a connected Lie subgroup of $G$ normalized by $H$.   Then the submanifold $Q/(H\cap Q) \,(=QH/H)$ is totally geodesic and is itself a G.O. manifold.

\end{lemma}

  \begin{remark} Let $M$ be a homogeneous manifold.   When expressing $M$ in the form $G/H$, we typically assume that $G$ acts effectively on $G/H$, i.e., that $H$ contains no non-trivial normal subgroups of $G$.   There are occasions, however, when it is convenient to relax that assumption.   In Lemma~\ref{lem.subman}, the action of $QH$ on $QH/H$ need not be effective.  $QH/H$ is a realization of the indicated submanifold as a (possibly non-effective) G.O. space. Of course, one can always obtain an effective realization as a G.O. space by dividing out by the effective kernel.
     \end{remark}

     \subsection{Structural resuls.}

 \begin{theorem}\label{thm.nil}  Let $(M,g)$ be a connected Riemannian geodesic orbit manifold and let $G/H$ be a realization of $M$
as a geodesic orbit space.  Let $\Lev(G)$ be a semisimple Levi factor of $G$, let $\Nil(G)$ be the nilradical of $G$, and define the subgroup $F$ of the normalizer of $G$ in $\Isom(M,g)$ and the action $\rho$ of $G\times F$ on $M$ as in Proposition~\ref{prop.gf}.   Let $R:=\Lev(G)\Nil(G)\times C(F)<G\times F$ where $C(F)$ is the center of $F$.  Then $R$ acts transitively and isometrically on $M$ via $\rho$.   The radical $\Nil(G)\times C(F)$ of $R$ is nilpotent of step size at most two.
 \end{theorem}

 \begin{proof} The fact that $R$ acts transitively and isometrically on $(M,g)$ is immediate from Corollary~\ref{cor.nilrad} and Lemma~\ref{lem.skew}.
It remains to be shown that the nilradical has step size at most two.
By Lemma~\ref{lem.subman}, $\Nil(G)$ is a G.O. nilmanifold.   By Theorem 2.2 of \cite{Gor96}, every G.O. nilmanifold is at most 2-step nilpotent.   Since $C(F)$ is abelian, it follows that $\Nil(G)\times C(F)$ is at most 2-step nilpotent.
 \end{proof}

\begin{remark}\label{rems} Theorem 1.15(i) in \cite{Gor96} claims, under the same hypotheses as Theorem~\ref{thm.nil}, that $\Lev(G)\Nil(G)$ acts transitively on the geodesic orbit manifold $M$.
As noted in the Introduction, this assertion is not true in full generality
(see e.~g. Examples 4 and 5 in~\cite{Nik2016}). On the other hand it is valid
for naturally reductive metrics, see Theorem 3.1 in~\cite{Gor85}.  Theorem~\ref{thm.nil} provides a correction to the version in~\cite{Gor96} for arbitrary G.O. manifolds.

%(ii) By Theorems 1.14 and 2.2 of \cite{Gor96}, both $\Lev(G)/(\Lev(G)\cap H)$ and $\Nil(G)$ with the metrics induced by $G$ are G.O. manifolds, and $\Nil(G)$ is at most two-step nilpotent.   Moreover, Theorem 1.15(ii) in the same article asserts that the noncompact part of $\Lev(G)$ commutes with $\Nil(G)$.
\end{remark}

%(ii) Given a homogeneous Riemannian manifold $G/H$, one may choose a semisimple Levi factor $\Lev(G)$ of $G$ normalized by $H$ such that $\Lev(G)H$ is a reductive subgroup of $G$.  In particular, $GH\cap \Rad(G)$ is abelian and commutes with $G$.  We will say that $\Lev(G)$ is \emph{compatible}with $H$.
%Let $\Lev{G}$ be a semisimple Levi factor compatible with $H$, and let $\Lev(G)=\Lev(G)_{\nc}\Lev(G)_c$ be the decomposition of $\Lev(G)$ into semisimple subgroups of noncompact and compact type, respectively.   (The intersection $\Lev(G)_{\nc}\cap \Lev{G}_c$ is discrete.)  Since the homomorphic projection $\pi:\mfg\to\Lev(\mfg)$ maps $\mfh$ into a compactly embedded subalgebra of $\Lev(G)$, we have $\pi(\mfh)<\mfk + \Lev(\mfg)_c$ where $\mfk$ is a maximal compactly embedded subalgebra of $\Lev(\mfg)_{\nc}$.  Let $K$ be the corresponding connected subgroup of $\Lev(G)_{\nc}$.

\begin{prop}\label{gnc}
Let $(G/H, g)$ be a connected Riemannian G.O. space and let $\Lev(G)$ be any Levi factor of $G$.  Then the noncompact part $\Lev(G)_{\nc}$ of $\Lev(G)$ is a normal subgroup of $G$, i.e. $\Lev(G)_{\nc}$ commutes with $\Rad(G)$.

\end{prop}

This Proposition is asserted in \cite{Gor96}, although the proof is not included there.   The proof in the naturally reductive case is given in \cite{Gor85}.
For completeness, we include the proof of the proposition here.

\begin{proof}
Write $\mfg=\mfh +\mfm$ as above and note that $\Nil(\mfg)<\mfm$.  Let $\mfo$ be the subalgebra of $\mfg$ defined by
\begin{equation}\label{eq.alg.o}
\mathfrak{o}=\{Y \in \mfg\,|\, \ad(Y)|_{\Nil(\mfg)}\mbox{\,\,is\,skew-symmetric} \}.
\end{equation}
(Here skew-symmetry is with respect to the inner product $\ip$ on $\mfm$.)
We first show that
\begin{equation}\label{eq.p}\mfg=\mfo+\Nil(\mfg).\end{equation}
   Let $\mfp$ be the orthogonal complement to $\Nil(\mfg)$ in $\mfm$ and let $Y\in\mfp$.     Given $X\in\Nil(\mfg)$, let $Z\in\mfh$ satisfy the conclusion of Lemma~\ref{G.O.-criterion}.   We have
      $$\la [Y, X],\,X\ra=\la [Y,X+Z],\,X\ra =0,$$
      where the first equation follows from the fact that $[\mfh,\mfp]<\mfp$ and the second equation is Lemma~\ref{G.O.-criterion}.
Thus $\mfp<\mfo$.   Since trivially $\mfh<\mfo$, Equation~(\ref{eq.p}) follows.

We next show that $\mfo$ contains a semisimple Levi factor of $\mfg$.  Let $\pi: \mfg \rightarrow \mathfrak{s}:=\mfg/ \Rad(\mfg)$ be the homomorphic projection.
By Equation~(\ref{eq.p}) and the fact that $\Nil(\mfg)<\Rad(\mfg)$, we have $\pi(\mfo)=\mfs$.
Since the Lie algebra $\mathfrak{s}$ is semisimple, Levi's Theorem (see e.g. Theorem~5.6.6 in \cite{HilNeeb}),
yields a homomorphism $\varphi: \mathfrak{s} \rightarrow \mathfrak{o}$ such that $\pi\circ \varphi =\operatorname{Id}_{\mathfrak{s}}$.
The  image of~$\varphi$ is the desired Levi factor $\wt{\Lev}(\mfg)$.

Let $\wt{\Lev}(\mfg)_{\nc}$ be the noncompact part of $\wt{\Lev}(\mfg)$.   The map $Y\mapsto\ad(Y)|_{\Rad(\mfg)}$ is a~representation of $\wt{\Lev}(\mfg)_{\nc}$ which acts by semisimple endomorphisms on the invariant subspace $\Nil(\mfg)$.
   Since the only representation of a semisimple Lie algebra of noncompact type by skew-symmetric endomorphisms is the trivial representation, we conclude that $[\wt{\Lev}(\mfg)_{\nc},\Nil(\mfg)]=0$.
 Finally since $[\wt{\Lev}(\mfg)_{\nc},\Rad(\mfg)]<\Nil(\mfg)$ and
since every representation of a semisimple Lie algebra is completely reducible, we have $[\wt{\Lev}(\mfg)_{\nc},\Rad(\mfg)]=0$.   In particular, $\wt{\Lev}(\mfg)_{\nc}$ is normal in $G$.

Finally, any semisimple Levi factor $\Lev(\mfg)$ is conjugate to $\wt{\Lev}(\mfg)$ via an element of $\Nil(\mfg)$.   Hence $\lgnc=\wt{\Lev}(\mfg)_{\nc}$ by normality of $\wt{\Lev}(\mfg)_{\nc}$ in $G$.
 \end{proof}
\smallskip

By Lemma~\ref{lem.subman}, if $G/H$ is a G.O. space, then $\Nil(G)$ is a G.O. nilmanifold.   We next discuss some properties of G.O. nilmanifolds.

\begin{lemma}[\cite{Wil}]\label{lem.wilson} Let $(N,g)$ be a Riemannian nilmanifold and $\ip$ the associated inner product on the Lie algebra $\mfn$.   Then
$\Isom(N,g)=N\rtimes H$, where $H=\Aut_{\orth}(N,g)$. {\rm (}See Notation~\ref{autorth}.{\rm )}
Thus the full isometry algebra of $(N,g)$ is the semi-direct sum $\mfn+\mfh$, where~$\mfh$ is the space of skew-symmetric derivations of $(\mfn,\ip)$.

\end{lemma}

\begin{notarem}\label{nota.vz} Let $(N,g)$ be a two-step Riemannian nilmanifold.
In the notation of Lemma~\ref{lem.wilson}, write $\mfn=\mfv+\mfz$ where $\mfz$ is the center of $\mfn$ and $\mfv=\mfz^\perp$ relative to $\ip$.
Note that each skew-symmetric derivation in $\mfh$ leaves each of $\mfv$ and $\mfz$ invariant.

\end{notarem}

\begin{prop}\label{proalgonilp1}  Suppose that $(N,g)$ is a G.O. nilmanifold, $\ip$ the associated inner product on $\mfn$, and $\mfh$ the
space of skew-symmetric derivations of $(\mfn,\ip)$.   We use the notation of~\ref{nota.vz}.
\begin{enumerate}
\item Let $\mfv=\mfv_1\oplus\cdots\oplus\mfv_k$ be an orthogonal decomposition of $\mfv$ into $\mfh$-invariant subspaces.
Then $[\mfv_i,\mfv_j]=\{0\}$ for all $i\neq j$.   In particular, $\mfv_i +[\mfv_i,\mfv_i]$ is an ideal in $\mfn$ for each~$i$.
{\rm(}These ideals are not disjoint in general; their derived algebras may overlap or even coincide.{\rm)}
\item Every $\mfh$-invariant subalgebra of $\mfn$ is an ideal.

\end{enumerate}
\end{prop}

\begin{proof}
(1)  It suffices to show that if $\mfv_1$ is an $\mfh$-invariant subspace, then $[\mfv_1,\mfv\ominus \mfv_1]=\{0\}$.   Let $X\in\mfv_1$ and $W\in \mfz$.     By Lemma~\ref{G.O.-criterion}, there exists $A\in\mfh$ such that $$\la [X+W+A,Y], X+W\ra=0$$
for all $Y\in\mfn$.
  Equivalently, since $W$ is central in $\mfn$ and $X\perp [\mfn,\mfn]$, we have
$$\la[X,Y],W\ra + \la [A,Y], X+W\ra=0.$$
For $Y\in\mfv\ominus \mfv_1$, the second term vanishes, so $\la [X,Y],W\ra=0$ for all $W\in \mfz$.   Hence, $[X,Y]=0$ and (1) follows.

(2) Let $\mfo$ be an $\mfh$-invariant subalgebra of $\mfn$.  Let $\pi:\mfn\to \mfv$ be the orthogonal projection with kernel $\mfz$ and let $\mfv_1=\pi(\mfo)$.   Since both $\mfo$ and $\mfv$ are $\mfh$-invariant, so is $\mfv_1$.  Noting that $[X,Y]=[\pi(X),\pi(Y)]$ for all $X,Y\in\mfn$, we have
$$[\mfn,\mfo]=[\mfv,\mfv_1]=[\mfv_1,\mfv_1]=[\mfo,\mfo]<\mfo$$
where the second equality follows from (1).
\end{proof}

 \subsection{G.O. manifolds that are diffeomorphic to $\R^n$.}\label{subsecrn}

\begin{notarem}\label{Iwasawa}
Recall that if $U$ is a connected semisimple Lie group of noncompact type and $K<U$ a subgroup such that $\mfk$ is a maximal compactly embedded subalgebra of $U$, then $U$ admits an Iwasawa decomposition $U=KS$, where $S$ is a simply-connected solvable Lie group and $K\cap S$ is trivial.   We will say that a solvable Lie group is an \emph{Iwasawa} group if it is a solvable Iwasawa factor  of some semisimple Lie group $U$.  In this case, $U/K$ with a left-invariant Riemannian metric is a Riemannian symmetric space.  \end{notarem}

\renewcommand{\theenumi}{\roman{enumi}}

\begin{theorem}\label{thm.solv} Let $(M,g)$ be an $n$-dimensional simply-connected G.O. manifold.   Then the following are equivalent:

\begin{enumerate}
\item  $(M,g)$ is a Riemannian solvmanifold; i.e., it admits a simply-transitive solvable group of isometries.

\item  $M$ is diffeomorphic to $\R^n$.

\item  In the notation of ~\ref{Iwasawa}, there exists an Iwasawa group $S$, a simply connected nilpotent Lie group $N$ commuting with $S$,
and a left-invariant metric $h$ on $S\times N$ so that $(S\times N, h)$ is isometric to $(M,g)$.   $N$ is at most 2-step nilpotent.
\end{enumerate}
\end{theorem}

\renewcommand{\theenumi}{\arabic{enumi}}

\begin{remark}  Every $n$-dimensional simply-connected solvable Lie group is diffeomorphic to $\R^n$, so the implication ``(i) implies (ii)" is trivial.  However the converse is not true for arbitrary homogeneous Riemannian manifolds as the following example illustrates.

\end{remark}

\begin{example}\label{ex.sl}  The universal cover $\t{G}$ of $\SL(2,\R)$ is diffeomorphic to $\R^3$.  For a \emph{generic} left-invariant metric $g$ on $\t{G}$, the identity component of the full isometry is just $\t{G}$ itself.  Thus $(\t{G},g)$ satisfies (ii) but not (i).

$\SL(2,\R)$ has Iwasawa decomposition $KS$ with $K=\SO(2)$ and $S$ the upper triangular matrices of determinant one.  The Iwasawa decomposition of $\t{G}$ is $\t{K}S$, where $\t{K}\simeq \R$ is the universal cover of $\SO(2)$.   If $g$ is a left-invariant metric that is also $\Ad(\t{K})$-invariant, then $\t{G}\times\t{K}$ acts almost effectively and isometrically on $(\t{G},g)$ and the simply-connected solvable Lie group $S\times\t{K}$ acts simply transitively.   Thus condition (i) does hold in this case.   As an aside, we note that the left-invariant metrics on $\t{G}$ that are also $\Ad(\t{K})$-invariant are precisely the naturally reductive metrics.  (See \cite{Gor85}.)
\end{example}

\begin{proof}[Proof of Theorem~\ref{thm.solv}] As already noted, the implication (i) implies (ii) is true even without the G.O. hypothesis.  The fact that (iii) implies (i) is trivial, so it remains to show that (ii) implies (iii).

Let $(G/H,g)$ be a realization of $(M,g)$ as a G.O. space.  Let $\mfl$ be a maximal compactly embedded subalgebra of $\mfg$ containing $\mfh$ and let $\Lev(\mfg)$ be a semisimple Levi factor such that $(\lg, \mfm,\mfl)$ is a compatible triple as in Definition~\ref{nota.compat}.   We have the following decompositions:
\begin{itemize}
\item[(a)] $\lg=\lgnc\oplus\lgc$, where $\lgnc$ and $\lgc$ are semisimple of noncompact type and compact type, respectively.
\item[(b)] $\mfl =(\mfl\cap \lg)+(\mfl\cap \Rad(\mfg))$ by Definition~\ref{nota.compat}.
\item[(c)] $\mfl\cap \lg=\mfk +\lgc$, where $\mfk$ is a maximal compactly embedded subalgebra of $\lgnc$ (by maximality of $\mfl$).
\item[(d)] $\lgnc=\mfk +\mfs$, an Iwasawa decomposition.
\end{itemize}

The assumption that $G/H$ is diffeomorphic to $\R^n$ implies that $H$  contains a maximal compact semisimple subgroup of $G$.   Thus, $[\mfl,\mfl]\subset \mfh$ and so
$\mfl <\mfh + C_\mfg(\mfh) =\mfh + C_\mfm(\mfh)$ with $C_\mfm(\mfh)$ abelian.  (Here as usual $\mfm$ denotes the orthogonal complement of $\mfh$ with respect to the Killing form of $\mfg$.)
   As in Lemma~\ref{lem.norm}, write
$C_{\mfm}(\mfh)=C(\mfg)\oplus\mff$
where, now, $\mff$ is abelian.
Lemma~\ref{lem.skew} implies that $\mfh + C_\mfm(\mfh)$ is a compactly embedded subalgebra of  $\mfg$.   Thus by maximality of $\mfl$, we have
\begin{equation}\label{ell}\mfl =\mfh \oplus C_{\mfm}(\mfh)=\mfh\oplus C(\mfg)\oplus\mff.\end{equation}
We emphasize, again by Lemma~\ref{lem.skew}, that $\ad_\mfg(X)|_\mfm$ is skew-symmetric for all $X\in \mfl$.

Set
\begin{equation}\label{u}\mfu=\mfs +\Nil(\mfg).\end{equation}
\noindent{\bf Claim.}  $\mfg$ can be written as a vector space direct sum of subalgebras
\begin{equation}\label{claim} \mfg=\mfh +\mfu+\mff.\end{equation}

\smallskip
To see that the sum on the right hand side of Equation~(\ref{claim}) is a vector space direct sum, note that by (b)--(d) and Equation~(\ref{u}), we have $\mfl\cap \mfu =\mfl\cap \Nil(\mfg)$.    For $X\in \mfl\cap \Nil(\mfg)$, the operator $\ad_\mfg(X)$ is both semisimple and nilpotent, hence trivial.  Thus $\mfl\cap \Nil(\mfg)=C(\mfg)$.  Hence $\mfl\cap \mfu =C(\mfg)$ and so Equation~(\ref{ell}) implies that $(\mfh\oplus \mff)\cap \mfu=\{0\}$, as claimed.

We next show that the right hand side of Equation~(\ref{claim}) exhausts $\mfg$.  By Equation~(\ref{ell}) and the fact that $C(\mfg)<\Nil(\mfg)<\mfu$, we have
\begin{equation}\label{h}\mfh+\mfu+\mff=\mfl+\mfu.\end{equation}
 Following the proof of Corollary~\ref{cor.nilrad} and using Equation~(\ref{ell}), we have
\begin{equation}\label{1}\mfg=\mfh+ C_{\mfm}(\mfh) + [\mfh,\mfm]=(\mfh+ C(\mfg) +\mff) + [\mfh,\mfm]=\mfl + [\mfh,\mfm]\end{equation}
and
\begin{equation}\label{2}[\mfh,\mfm]< \Lev(\mfg)+\Nil(\mfg).\end{equation} By (b)--(d) and Equation~(\ref{u}), $\lg < \mfl +\mfu$.   Thus the claim follows
from Equations~(\ref{u}), (\ref{h}), (\ref{1}), and (\ref{2}).

We now apply Proposition~\ref{prop.gf}(4) with $\mfu=\mfs +\Nil(\mfg)$ and $\mfv=\mff$ to conclude that the Lie group
$U \times F$ acts transitively and isometrically on $(G/H,g)$, where $U=S\Nil(G)$ is the connected subgroup of $G$ with Lie algebra $\mfu$ and where $F$ acts by right translations.   The fact that $\mfh\cap (\mfu +\mff)$ is trivial implies that the isotropy subgroup of $U\times F$ is discrete.  Since $M$ is simply-connected, we conclude that the isotropy group is trivial, i.e., the action is simply transitive.   Simple-connectivity also implies that $S\cap \Nil(G)$ is trivial.  Moreover, by Proposition~\ref{gnc}, $\lgnc$ and hence $\mfs$ commutes with $\Rad(\mfg)$.  Thus $U$ is a Lie group direct product of $U=S\times \Nil(G)$.   Setting $N=\Nil(G)\times F$, we thus have $U\times
F=S\times N$.    Recalling that $\mff$ and hence $F$ is abelian, we see that $N$ is nilpotent.  As in Theorem~\ref{thm.nil}, $N$ has step size at most two.
\end{proof}

\begin{theorem}\label{thm.submers} Assume that $(M,g)$ satisfies the equivalent conditions of Theorem~\ref{thm.solv}.  Then $(M,g)$ is isometric to one of the following:
\begin{enumerate}
\item a Riemannian symmetric space of non-compact type;
\item a simply-connected Riemannian G.O. nilmanifold $(N,h)$ (necessarily of step size at most 2);
\item the total space of a Riemannian submersion $\pi:  M\to P$ with totally geodesic fibers, where the base is a Riemannian symmetric space of noncompact type and the fibers are isometric to a simply-connected Riemannian G.O. nilmanifold.\end{enumerate}
\end{theorem}

\begin{proof}[Proof of Theorem~\ref{thm.submers}] In the proof of Theorem~\ref{thm.solv}, we started with an arbitrary realization $(G/H,g)$ of $(M,g)$ as
a G.O. space.    We now assume that $G$ is the identity component of the full isometry group of $(M,g)$ and continue to use the notation of the proof
of Theorem~\ref{thm.solv}.  We must now have $\mff=0$ since, otherwise, Proposition~\ref{prop.gf} would give us a larger isometry group.
Consequently, the maximal compactly embedded subalgebra~$\mfl$ in the proof satisfies
\begin{equation}\label{eq.l}\mfk+\lgc<\mfl =\mfh+C(\mfg)<\mfh+\Nil(\mfg)<\mfk+\lgc+\Rad(\mfg)\end{equation}
 by Equation~(\ref{ell}) and Equations (a)--(c) in the proof of Theorem~\ref{thm.solv}.  Moreover, by Equations~(\ref{u}) and (\ref{claim}), we have
 \begin{equation}\label{eq.hsn}\mfg=\mfh+\mfs+\Nil(\mfg),\end{equation}
 vector space direct sum.

First consider the case that $G$ is semisimple, i.e., $G=\Lev(G)$.   By Equation~(\ref{eq.l}) and the fact that $C(\mfg)=0$ in this case, we see that $\mfh$ is a maximal compactly embedded subalgebra of $\mfg$.   Consequently, $(G/H,g)$ must be a symmetric case of noncompact type. This is case (1) of the theorem.

Next consider the case that $\Lev(G)_{\nc}$ is trivial.    Then $S$ is trivial, so Theorem~\ref{thm.solv} immediately tells us that $(M,g)$ is isometric to a G.O. nilmanifold, giving us case (2).

We now consider the general case.  Let $Q=\Nil(G)H$.   By Lemma~\ref{lem.subman}, $Q/H$ is a totally geodesic submanifold of $G/H$ and is a G.O. manifold.    $\Nil(G)$ acts simply transitively on $Q/H$, so $Q/H\simeq (\Nil(G),h)$ where $h$ is the left-invariant metric on $\Nil(G)$ induced by $g$.

By Equations/Inequalities~(\ref{eq.l}) and (\ref{eq.hsn}), the Lie algebra $\mfq$ of $Q$ must be given by
\begin{equation}\label{eq.q}\mfq=\mfk+\lgc+\Rad(\mfg)\end{equation}
(Indeed, $\mfq$ is contained in the right hand side by Equation~(\ref{eq.l}); it must exhaust the right hand side by Equation~(\ref{eq.hsn}).)

Since $Q$ is a closed subgroup of $G$, the quotient $G/Q$ is a manifold and we have a submersion $G\to G/Q$.   Since $H <Q$, we get an induced submersion
$$
\pi:G/H\to G/Q.
$$
By Equation~(\ref{eq.q}), we can identify $G/Q$ with $\lGnc/K$, where $K$ is the connected subgroup of $\lGnc$ with Lie algebra $\mfk$.
Under the identification of $G/H$ with $M$, we thus have a submersion
$$
\pi:M\to \lGnc/K.
$$
The metric induced on the base by the metric $g$ on $M$ is left $\lGnc$-invariant and thus the base is a symmetric space of noncompact type.   The fibers are isometric to $Q/H$ with the induced metric and thus to $(\Nil(G),h)$.  They are totally geodesic by Lemma~\ref{lem.subman}.

It remains to be shown that $\pi$ is a Riemannian submersion.
Letting $\mfp$ be the orthogonal complement to $\mfk$ in $\Lev(\mfg)_{\nc}$ relative to the Killing form of $\lgnc$ (which is the same as the restriction to $\lgnc$ of the Killing form of $\mfg$), then $\lgnc=\mfk+\mfp$ is a Cartan decomposition of $\lgnc$.  The orthogonal complement $\mfm$ of $\mfh$ in $\mfg$ satisfies $\mfm=\mfp+\Nil(\mfg)$.  Since $\mfk<\mfl$, the proof of Theorem~\ref{thm.solv} shows that $\ad_\mfg(X)|_\mfm$ is skew-symmetric for all $X\in\mfk$.   By Proposition~\ref{gnc}, $[\mfk,\Nil(\mfg)]=0$.   Since $[\mfk,\mfp]=\mfp$, it follows that
$\mfp\perp \Nil(\mfg)$ and thus that $\pi$ is a Riemannian submersion.
\end{proof}

\begin{remark} In the statement of Theorem~\ref{thm.solv}, the metric on $S\times N$ need not be a product metric, and the induced metric on $S$ need
not be symmetric.    We only have that $S$ is a global section of the submersion in Theorem~\ref{thm.submers} and $N$ is isometric to the fiber.
As such, $S$ is diffeomorphic to the base but the induced metric on $S$ is in general not isometric to that of the base.

To illustrate this,
let us consider the naturally reductive manifold in Example \ref{ex.sl}, where the universal cover $\t{\SL(2,\R)}$ of $\SL(2,\R)$ is supplied with
a left-invariant and $\Ad(\t{K})$-invariant  Riemannian metric $g$ and the group $\t{\SL}(2,\R)\times \t{K}$ acts transitively and almost effectively.
In the language of the proof of Theorem~\ref{thm.solv},  $\t{K}$ plays the role of $F$ while $\Nil(G)$ is trivial, and so $N=F=\t{K}$.  We will henceforth
refer to the second factor of $\t{\SL}(2,\R)\times \t{K}$ as $F$, while keeping in mind that it is a copy of $\t{K}$.  There are two ways to see that the
Riemannian metric is not the product of the symmetric space metric (i.e., the hyperbolic metric) on the Iwasawa group $S$ and the Euclidean
metric on the one-dimensional group $F\simeq\R$.   If it were the product metric, then the identity component of the full isometry group would
be $\operatorname{PSL}(2,\R)\times F$, whereas the identify component of the full isometry group of the metric in Example~\ref{ex.sl}
is $(\t{\SL}(2,\R)\times F)/D$, where the effective kernel $D$ is the discrete center of $\t{\SL}(2,\R)$ diagonally
embedded in $\t{K}\times F < \t{\SL}(2,\R)\times F$.  To understand the metric more directly, let $\mathfrak{sl}(2,\R)=\mfk+\mfp$ be
the Cartan decomposition.   The isotropy subalgebra $\mfh$ of the full isometry algebra $\mathfrak{sl}(2,\R) \oplus\mff$ is a diagonally
embedded copy of $\mfk$ (recall that $\mff\simeq\mfk$).  The metric is naturally reductive with respect to the decomposition $\mfh +\mfm$,
where $\mfm = \mfp \oplus\mff$.   We do have $\mfp\perp \mff$, but $\mfm$ is identified with the tangent space only at the basepoint.
Now consider the metric induced on the simply-transitive group $S\times F$.   For $X\in\mfs$, write $X=X_\mfh +(X_\mfp + X_\mff)$,
decomposition of $X$ with respect to the decomposition $\mfh +\mfm=\mfh +(\mfp+\mff)$ of the full isometry algebra.
Observe that for those elements $X\in\mfs$ that do not lie in $\mfp$, the component $X_\mff$ is non-trivial.  Thus $\mfs$ is not orthogonal
to $\mff$ and the induced metric on $S$ is not the hyperbolic one.

\end{remark}

\medskip

{\bf Acknowlegment.}
The authors are grateful to the referee for helpful comments and suggestions that improved the presentation of this paper.

\vspace{-1mm}

\bibliographystyle{amsunsrt}

\vspace{1mm}

\end{document}